\documentclass[reqno,11pt]{amsart}

\usepackage{amsmath}
\usepackage{amssymb}
\usepackage{latexsym}
\usepackage{tabmac}
\usepackage{mathdots}
\usepackage{booktabs}

\allowdisplaybreaks[3]

\numberwithin{equation}{section}

\newtheorem{lemma}{Lemma}[section]
 \newtheorem{prop}[lemma]{Proposition}
\newtheorem{thm}[lemma]{Theorem}
\newtheorem{cor}[lemma]{Corollary}
\newtheorem{dfn}[lemma]{Definition}

\newcommand{\C}{{\mathbb C}}
\newcommand{\Z}{{\mathbb Z}}

\newcommand{\Sptn}{ {\mathrm{Sp}}_{2n} }
\newcommand{\GL}{{ \mathrm{GL}} }
\newcommand{\Hom}{{ \mathrm{Hom}} } 

\newcommand{\twedge}{{\textstyle \bigwedge}}

\begin{document}

\title[A reciprocity law and the skew Pieri rule for the symplectic group]{
A reciprocity law and the skew Pieri rule for the symplectic group}

\author {Roger Howe}
\address{
Department of Teaching, Learning and Culture,
308 Harrington Tower,
4232 Texas A\&M University,
College Station,   TX  77843-4232 ,
USA.}

\email{ roger.howe@yale.edu}

\author {Roman L\' avi\v cka}
\address{Charles University, Faculty of Mathematics and Physics, Sokolovsk\'a 83, 186 75 Praha, Czech Republic}

\email{  lavicka@karlin.mff.cuni.cz }

\author {Soo Teck Lee }

 \address {Department of Mathematics, National University of Singapore,
Block S17, 10 Lower Kent Ridge Road, Singapore 119076, Singapore.}

 \email {matleest@nus.edu.sg}

\author {Vladim\'ir Sou\v cek }
\address{Charles University, Faculty of Mathematics and Physics, Sokolovsk\'a 83, 186 75 Praha, Czech Republic}

\email{   soucek@karlin.mff.cuni.cz}

\begin{abstract} We use the theory of skew duality to
show that decomposing the tensor product of $k$ irreducible representations 
 of the symplectic group $\mathrm{Sp}_{2m}=\mathrm{Sp}_{2m}(\C)$ is equivalent to branching from $\mathrm{Sp}_{2n}$ to $\mathrm{Sp}_{2n_1}\times\cdots\times\mathrm{Sp}_{2n_k}$ where $n,n_1,...,n_k$ are positive integers such that $n=n_1+\cdots+n_k$ and  the $n_j$s depend on $m$ as well as the representations in the tensor product. Using this result and a work of J. Lepowsky, we obtain a {\em skew Pieri rule} for $\mathrm{Sp}_{2m}$, i.e.,
a description of the irreducible decomposition of
the tensor product of an irreducible representation of the symplectic group $\mathrm{Sp}_{2m}$ with a fundamental representation.

\end{abstract}

\maketitle

\section{Introduction}
 It is well known that decomposing the tensor product of two irreducible polynomial representations of the general linear group $\mathrm{GL}_n=\mathrm{GL}_n(\C)$ is equivalent to the branching of a polynomial representation of $\mathrm{GL}_{k+\ell}$ to $\mathrm{GL}_k\times\mathrm{GL}_\ell$, where $k$ and $\ell$ depend on $n$ as well as the representations in the tensor product.  
This is one of a family of reciprocity laws in classical invariant theory  (\cite{How}). 

\medskip
More generally, using the theory of dual pair correspondences and see-saw dual pairs, one can obtain reciprocity theorems for  the branching rules for certain classical symmetric pairs  (\cite{HTW1, HTW}). 
Except for the case of $\mathrm{GL}_n$ described above,
these reciprocity theorems relate the  branching rule for finite dimensional representations of a symmetric pair to the branching rule for certain  infinite dimensional representations of another symmetric pair. For example, decomposing the tensor product of two irreducible rational representations of the symplectic group 
 $\mathrm{Sp}_{2n}=\mathrm{Sp}_{2n}(\C)$
is equivalent to the branching of certain irreducible holomorphic representations of $\mathfrak{so}_{p+q}$ to $\mathfrak{so}_p\oplus \mathfrak{so}_q$, where $p$ and $q$ depends on $n$ and the representations in the tensor product. 

\medskip
In this paper, we use the theory of skew duality (\cite{How}) to obtain a reciprocity law for the symplectic group that involves only finite dimensional representations. It  relates two sets of branching rules. Similar but somewhat more complicated results can be formulated for orthogonal groups.

\medskip
Recall that the irreducible rational representations of $\mathrm{Sp}_{2n}$ are indexed by Young diagrams $D$ with at most $n$ rows. We denote the rational representation of $\mathrm{Sp}_{2n}$ corresponding to $D$ by $\tau^D_{2n}$. 
For positive integers $n$ and $m$, let $\mathcal{R}_{n,m}$ be the set of all Young diagrams $D$ such that $D$ has at most $n$ rows and at most $m$ columns. Let $\iota_{n,m}:\mathcal{R}_{n,m}\to\mathcal{R}_{m,n}$ be the involution  defined in equation \eqref{iotanm}.

\medskip\noindent {\bf MAIN THEOREM.} {\em
  Let $n,n_1,...,n_k,m$ be positive integers such that $n=n_1+\cdots+n_k$, and let $D_1,...,D_k,E$ be Young diagrams such that 
$D_i\in\mathcal{R}_{n_i,m}$ for $1\le i\le k$ and 
  $E\in\mathcal{R}_{n,m}$. Then
\[\dim\mathrm{Hom}_H\left(\bigotimes_{i=1}^k \tau^{D_i}_{2n_i} ,\tau^E_{2n}\right)=\dim\mathrm{Hom}_{\mathrm{Sp}_{2m}}\left(\tau^{\iota_{n,m}(E)}_{2m},\bigotimes_{i=1}^k \tau^{\iota_{n_i,m}(D_i)}_{2m}\right ),\]
where $H=\mathrm{Sp}_{2n_1}\times\cdots\times\mathrm{Sp}_{2n_k}$ is regarded as a subgroup of $\mathrm{Sp}_{2n}$.
}

\medskip For $1\le r\le m$, let $\omega_r=\tau^{1^r}_{2m}$ where $1^r$ is the Young diagram which has  exactly $1$ column of length $r$ (so it has a total of $r$ boxes). Then $\omega_1,...,\omega_m$ are the fundamental representations of $\mathrm{Sp}_{2m}$. We will call a description of how a tensor product of the form $\tau^D_{2m}\otimes\omega_r$ decomposes into irreducible representations, a {\em skew Pieri rule}
for $\mathrm{Sp}_{2m}$.
According to the Main Theorem, we see that the skew Pieri rule for $\mathrm{Sp}_{2m}$ is equivalent to the branching rule from $\mathrm{Sp}_{2n+2}$ to $\mathrm{Sp}_{2n}\times \mathrm{Sp}_2$ where $n$ depends on $m$ and the representations in the tensor product. By using this fact and a result of J. Lepowsky (\cite{Lep}), we obtain a new skew Pieri rule for $\mathrm{Sp}_{2m}$.
For other known approaches, see e.g.\ \cite{HTW, KT, Lit, Nak, Sun, Wey}.

 \medskip This paper is arranged as follows. In Section 2, we introduce
 notation for representations of $\mathrm{Sp}_{2m}$. In Section 3, we
 describe the basic tool we are using from \cite{How} - the skew duality for the group
 $\mathrm{Sp}_{2m}$. Section 4 is the core of the paper, it contains the proof of the main theorem. As an application, we deduce a new version of the skew Pieri rule in Section 5.

 \bigskip
\begin{center}ACKNOWLEDGEMENT\end{center}
The first and the third-named authors are grateful to Charles University  for  warm hospitality and for providing a 
productive working atmosphere during their visit in Summer, 2016, which was partially supported by the grant GACR P201-12-G028.
The fourth author was also supported by the grant GACR P201-12-G028.

\section{Preliminaries}\label{preliminaries} 
Let $\mathrm{Sp}_{2n}=\mathrm{Sp}_{2n}(\C)$ be the subgroup of 
$\GL_{2n}(\C)$ which preserves
the symplectic form $(\cdot,\cdot)$ on $\C^{2n}$  given by
\begin{equation*}
( \begin{pmatrix} u_1 \\ \vdots \\ u_{2n}  \end{pmatrix},
  \begin{pmatrix} v_1\\ \vdots \\ v_{2n} \end{pmatrix} )
= \sum_{j=1}^n ( u_j v_{n+j} - u_{n+j} v_j).
\end{equation*}
In this section, we shall review some basic facts about  the representations of 
$\mathrm{Sp}_{2n}$.

\medskip
 First, recall that a {\em Young diagram} $D$
is an array of square boxes arranged in left-justified horizontal rows,
with each row no longer than the one just above it.
\[D=\  \tableau[s]{&&&&&\\
&&&\\
&&&\\
&\\}
\]
If $D$ has
at most $r$ rows, then we shall denote it by
\begin{equation*}
D=(d_1,...,d_r)
\end{equation*}
where for each $i$, $d_i$ is the number of boxes in the $i$-th row
of $D$. For the Young diagram above, we have $D=(6,4,4,2)$. 
In particular, $(d)$ is the Young diagram with one row and $d$ boxes. For a positive integer $c$, let $1^c$ denote the Young diagram with only one column and $c$ boxes, i.e.
\[1^c=(\overbrace{1,...,1}^c).\]
The {\em depth} of $D$, denoted by $\ell(D)$, is the number of
nonzero rows in $D$. The number 
$\sum_{i \geq 1} d_i$, which is the total number of boxes in $D$, will be denoted by $|D|$. 

\medskip
Let $A_{2n}$ be the diagonal torus of $\Sptn$, and let $U_{2n}$ be the maximal unipotent subgroup of $\Sptn$.
The  irreducible rational representations of $\Sptn$ are parametrized
by Young diagrams with at most $n$ rows. If $D$ is such a Young diagram, we
shall denote the corresponding $\Sptn$ representation by $\tau^D_{2n}$, and
denote its highest weight by $\psi^D_{2n}:A_{2n}\to \C^\times$.  
See Section 3.8 of \cite{How} for more details.

\section{Skew duality for the symplectic groups}

Our formulation of the duality theorem makes use of two involutions on the set of Young diagrams.
The first one is {\it conjugation}, $D \longrightarrow D^t$, which flips a diagram across the diagonal, turning the rows into columns, 
and the columns into rows. For example, $(d)^t=1^d$.
Note that  $\ell(D^t)$ is the number of columns in $D$, which equals the length of the first row of $D$.

For positive integers $n$ and $m$, let
\[{\mathcal R}_{n,m}:=\{D:\ \mbox{$D$ is a Young diagram with $\ell(D)\le n$ and $\ell(D^t)\le m$}\}.\]
For $D=(d_1,\ldots,d_n)\in{\mathcal R}_{n,m}$, we can describe the conjugate diagram $D^t$ explicitly as
$$
D^t=(\overbrace{\underbrace{n,\ldots,n}_{d_n},\underbrace{n-1,\ldots,n-1}_{d_{n-1}-d_n},
\ldots,\underbrace{1,\ldots,1}_{d_1-d_2},
\underbrace{0,\ldots,0}_{m-d_1}}^{m}).\;\;
 $$ 
Note that the conjugation maps ${\mathcal R}_{n,m}$ to ${\mathcal R}_{m,n}$.

Define an involution $r_{n,m}$ on ${\mathcal R}_{m,n}$. For $E=(a_1,...,a_m)\in\mathcal{R}_{m,n}$, define the involution
\begin{equation}\label{rnm}
r_{n,m}(E)=(n-a_m,n-a_{m-1},...,n-a_1).
\end{equation}

 \begin{dfn}\label{df_involutions}
 The map $\iota_{n,m}:\mathcal{R}_{n,m}\to \mathcal{R}_{m,n}$ is defined by
\begin{equation}\label{iotanm}
\iota_{n,m}(D):=r_{n,m}(D^t).
\end{equation}
Denote its inverse $j_{n,m}=\iota_{n,m}^{-1}$.
\end{dfn}

It is easy to see that, for $D=(d_1,\ldots,d_n)\in{\mathcal R}_{n,m},$ we have
\begin{equation}\label{i}
 \iota_{n,m}(D)=
 (\overbrace{\underbrace{n,\ldots,n}_{m-d_1},\underbrace{n-1,\ldots,n-1}_{d_1-d_2},\ldots,\underbrace{1,\ldots,1}_{d_{n-1}-d_n},
\underbrace{0,\ldots,0}_{d_n}}^{m}).
\end{equation}
For $E=(e_1,\ldots,e_m)\in \mathcal{R}_{m,n},$ we have 
$
j_{n,m}(E)=(r_{n,m}(E))^t
$
and
\begin{equation}\label{j}
 j_{n,m}(E)=
 ( \overbrace{\underbrace{m,\ldots,m}_{n-e_1},\underbrace{m-1,\ldots,m-1}_{e_1-e_2}
 \ldots,\underbrace{1,\ldots,1}_{e_{m-1}-e_m},
\underbrace{0,\ldots,0}_{e_m}}^{n}).
\end{equation}

 In geometric terms, the conjugation of a diagram corresponds
 to the diagram reflected along the main diagonal, while
 the involution $r_{n,m}$ amounts to the taking the complement of the diagram in the box of size
  $m\times n$, and rotating it
 by 180 degrees around the center of the box. For example: 
\[D=\  \tableau[s]{&&&&\\
&\\
&\\
\\}
\ \ \ \
D^t=\  \tableau[s]{&&&\\
&&\\
\\
\\
\\}
\ \ \ \
\iota_{6,5}(D)=\  \tableau[s]{&&&&\\
&&&&\\
&&&&\\
&&\\
&\\}
\]
 

\vskip .4 in

\medskip Next, we let
 $V_{n,m}:=\C^{2n}\otimes \C^m$ and consider the exterior algebra 
$\bigwedge(V_{n,m})$ on $V_{n,m}$. The symplectic group $\mathrm{Sp}_{2n}$ acts on $V_{n,m}$ via its action on the first factor, and this action is extended in the usual way to $\bigwedge(V_{n,m})$. Let $\mathrm{End}(V_{n,m})$ be the algebra of all endomorphisms of $V_{n,m}$. By \cite{How},  the algebra $\mathrm{End}_{\mathrm{Sp}_{2n}}(V_{n,m})$ of all endomorphisms which commute with the action by $\mathrm{Sp}_{2n}$ is generated by a Lie algebra isomorphic to $\mathfrak{sp}_{2m}=
\mathfrak{sp}_{2m}(\C)$. We shall abuse notation and denote this Lie algebra by $\mathfrak{sp}_{2m}$. Then $\bigwedge(V_{n,m})$ is a module for $\mathrm{Sp}_{2n}\times\mathfrak{sp}_{2m}$.
Since $\mathrm{Sp}_{2m}$ is connected and simply connected, a representation of $\mathfrak{sp}_{2m}$ gives rise to a unique representation of $\mathrm{Sp}_{2m}$ by exponentiation. Thus $\bigwedge(V_{n,m})$ is also a $\mathrm{Sp}_{2n}\times\mathrm{Sp}_{2m}$ and its structure is given by the following theorem:

\medskip\begin{prop}\label{howethm} {\rm (\cite[Theorem 3.8.9.3]{How})} We have the $\mathrm{Sp}_{2n}\times\mathrm{Sp}_{2m}$-module decomposition
\[{\textstyle\bigwedge}(V_{n,m})=\bigoplus_{D\in\mathcal{R}_{n,m}}\tau^D_{2n}\otimes\tau^{\iota_{n,m}(D)}_{2m}.\]
\end{prop}

\medskip\noindent In particular, by taking $n=1$, we  obtain the $\mathrm{Sp}_{2}\times\mathrm{Sp}_{2m}$-module decomposition
\[\twedge(V_{1,m})=\bigoplus_{0\le j\le m}\tau^{(j)}_{2}\otimes\tau^{1^{m-j}}_{2m}.\]
 

\section{The proof of the main theorem}
In this section, we shall prove the Main Theorem, as stated in Introduction.


\medskip
\noindent{\em Proof of the Main Theorem.}
Let
 \[G=\mathrm{Sp}_{2n},\quad H=\mathrm{Sp}_{2n_1}\times\mathrm{Sp}_{2n_2}\times\cdots\times\mathrm{Sp}_{2n_k},\]
\[U_G=U_{2n},\quad U_H=U_{2n_1}\times U_{2n_2}\times\cdots\times U_{2n_k} ,\]
\[A_G=A_{2n},\quad A_H=A_{2n_1}\times A_{2n_2}\times\cdots\times A_{2n_k}.\]
We consider the exterior algebra $\bigwedge(V_{n,m})$. It is a module for  $G\times\mathrm{Sp}_{2m}$, and by Proposition \ref{howethm},  it can be decomposed as 
\begin{equation}\label{wv1}\twedge(V_{n,m})=\bigoplus_{E\in\mathcal{R}_{n,m}}\tau^E_{2n}\otimes\tau^{\iota_{n,m}(E)}_{2m}.
\end{equation}
On the other hand, we can write $V_{n,m}$ as a direct sum
\[V_{n,m}=\C^{2n}\otimes\C^m=\left(\bigoplus_{i=1}^k\C^{2n_i}\right)\otimes\C^m\cong\bigoplus_{i=1}^k\left(\C^{2n_i}\otimes\C^m\right),\]
so that
\begin{equation}
\twedge(V_{n,m})\cong\bigotimes_{i=1}^k\twedge( \C^{2n_i}\otimes\C^m).\label{wdecom1}
\end{equation}
 We see from equation \eqref{wdecom1}  that 
$\bigwedge(V_{n,m})$  is also
 a module for 
\[(\mathrm{Sp}_{2n_1}\times\mathrm{Sp}_{2m})\times
(\mathrm{Sp}_{2n_2}\times\mathrm{Sp}_{2m})\times\cdots\times
(\mathrm{Sp}_{2n_k}\times\mathrm{Sp}_{2m}) \cong H\times \mathrm{Sp}_{2m}^{k}. \]
We shall restrict this action to $H\times \mathrm{Sp}_{2m}$, where $\mathrm{Sp}_{2m}$ is identified with the diagonal subgroup of $\mathrm{Sp}_{2m}^{k}$.

\medskip
By equation \eqref{wdecom1} and Proposition \ref{howethm},  $\bigwedge(V_{n,m})$ can be decomposed as an $H\times\mathrm{Sp}_{2m}$ module as
\begin{eqnarray}
\twedge(V_{n,m}) &\cong& \bigotimes_{i=1}^k\left(
\bigoplus_{D_i\in\mathcal{R}_{n_i,m}} \tau^{D_i}_{2n_i}\otimes\tau^{\iota_{n_i,m}(D_i)}_{2m}\right)\nonumber\\
&\cong&
\bigoplus_{D_i\in\mathcal{R}_{n_i,m},1\le i\le k}
\left\{
\left(\bigotimes_{i=1}^k \tau^{D_i}_{2n_i}\right)\otimes
\left(\bigotimes_{i=1}^k \tau^{\iota_{n_i,m}(D_i)}_{2m}  \right)\right\}.
 \label{wv2}
\end{eqnarray}

\medskip
We now consider the algebra 
$\twedge(V_{n,m})^{ U_H\times U_{2m}}$ of
$U_H\times U_{2m}$
 invariants in $\twedge(V_{n,m})$. It is a module for $A_H \times A_{2m}$. For Young diagrams $D_i\in\mathcal{R}_{n_i,m}$ for $1\le i\le k$, $E\in\mathcal{R}_{n,m}$, denote by
$W_{(D_1,...,D_k,E)}$ the space of $w\in\twedge(V_{n,m})^{ U_H\times U_{2m}}$ such that
\[(a,b).w=\psi^{(D_1,...,D_k)}_{H}(a )\psi^{\iota_{n,m}(E)}_{2m}(b)w\
\forall (a,b)\in A_H\times A_{2m}\]
where $\psi^{(D_1,...,D_k)}_H=\psi^{D_1}_{2n_1}\times\cdots\times
\psi^{D_k}_{2n_k}$.
Then $\twedge(V_{n,m})^{ U_H\times U_{2m}}$ can be
  decomposed as a direct sum
\[\twedge(V_{n,m})^{ U_H\times U_{2m}}=\bigoplus_{D_1,...,D_k,E}W_{(D_1,...,D_k,E)}\]
taken over all Young diagrams
$D_i\in\mathcal{R}_{n_i,m}$ $(1\le i\le k)$
  and $E\in\mathcal{R}_{n,m}$.

\medskip Using formula \eqref{wv1}, we obtain
\begin{eqnarray*}
\twedge(V_{n,m})^{ U_H\times U_{2m}}
&=&\bigoplus_{E\in\mathcal{R}_{n,m}}(\tau^E_{2n})^{U_H}\otimes
(\tau^{\iota_{n,m}(E)}_{2m})^{U_{2m}}\\
&=&\bigoplus_{E\in\mathcal{R}_{n,m}}
\left(\bigoplus_{D_1,...,D_k}
(\tau^E_{2n})^{U_H}_{D_1,...,D_k}\right)\otimes
(\tau^{\iota_{n,m}(E)}_{2m})^{U_{2m}},
\end{eqnarray*}
where $(\tau^E_{2n})^{U_H}_{D_1,...,D_k}$ is the space of $H$ highest weight vectors of weight $\psi^{(D_1,...,D_k)}_H$ in $\tau^E_{2n}$.
From this, we see that 
\[W_{(D_1,...,D_k,E)}=(\tau^E_{2n})^{U_H}_{D_1,...,D_k}\otimes
(\tau^{\iota_{n,m}(E)}_{2m})^{U_{2m}}.\]
 Since $\dim (\tau^{\iota_{n,m}(E)}_{2m})^{U_{2m}}=1$,
\begin{eqnarray}
\dim W_{(D_1,...,D_k,E)}&=&\dim( \tau^E_{2n})^{U_H}_{D_1,...,D_k}\dim
(\tau^{\iota_{n,m}(E)}_{2m})^{U_{2m}}=\dim( \tau^E_{2n})^{U_H}_{D_1,...,D_k}\nonumber\\
&=&\dim\mathrm{Hom}_H\left(\bigotimes_{i=1}^k \tau^{D_i}_{2n_i} ,\tau^E_{2n}\right),\label{dw1}
\end{eqnarray}
which is the multiplicity of $\bigotimes_{i=1}^k \tau^{D_i}_{2n_i} $ in $\tau^E_{2n}$.

\medskip Next, we use formula \eqref{wv2} to obtain

\begin{eqnarray*}
\twedge(V_{n,m})^{U_H\times U_{2m}} &\cong& 
\bigoplus_{D_i\in\mathcal{R}_{n_i,m},1\le i\le k}
\left\{
\left(\bigotimes_{i=1}^k \tau^{D_i}_{2n_i}\right)^{U_H}\otimes
\left(\bigotimes_{i=1}^k \tau^{\iota_{n_i,m}(D_i)}_{2m}  \right)^{U_{2m}}\right\}\\
&\cong& 
\bigoplus_{D_i\in\mathcal{R}_{n_i,m},1\le i\le k}
\left\{
\left(\bigotimes_{i=1}^k \tau^{D_i}_{2n_i}\right)^{U_H}\otimes
\left[\bigoplus_{E\in\mathcal{R}_{n,m}}
\left(\bigotimes_{i=1}^k \tau^{\iota_{n_i,m}(D_i)}_{2m}  \right)^{U_{2m}}_E\right]\right\}\\
\end{eqnarray*}
 where 
$\left(\bigotimes_{i=1}^k \tau^{\iota_{n_i,m}(D_i)}_{2m}  \right)^{U_{2m}}_E   $
 is the space of $\mathrm{Sp}_{2m}$ highest weight vectors of weight $\psi^{\iota_{n,m}(E)}_{2m}$ in
$\bigotimes_{i=1}^k \tau^{\iota_{n_i,m}(D_i)}_{2m}$.
 From this, we see that 
\[W_{(D_1,...,D_k,E)}
=\left(\bigotimes_{i=1}^k \tau^{D_i}_{2n_i}\right)^{U_H}\otimes
 \left(\bigotimes_{i=1}^k \tau^{\iota_{n_i,m}(D_i)}_{2m}  \right)^{U_{2m}}_E.\]
 Since $\dim \left(\bigotimes_{i=1}^k \tau^{D_i}_{2n_i}\right)^{U_H}    =1$,
\begin{eqnarray}
\dim W_{(D_1,...,D_k,E)}&=&\dim\left(\bigotimes_{i=1}^k \tau^{D_i}_{2n_i}\right)^{U_H}   \dim \left(\bigotimes_{i=1}^k \tau^{\iota_{n_i,m}(D_i)}_{2m}  \right)^{U_{2m}}_E  
=\dim \left(\bigotimes_{i=1}^k \tau^{\iota_{n_i,m}(D_i)}_{2m}  \right)^{U_{2m}}_E  
 \nonumber\\
&=&\dim\mathrm{Hom}_{\mathrm{Sp}_{2m}}\left(\tau^{\iota_{n,m}(E)}_{2m},\bigotimes_{i=1}^k \tau^{\iota_{n_i,m}(D_i)}_{2m}\right )   \label{dw2}
\end{eqnarray}
which is the multiplicity of $\tau^{\iota_{n,m}(E)}_{2m}$ in $\bigotimes_{i=1}^k \tau^{\iota_{n_i,m}(D_i)}_{2m}$.

\medskip Finally, by equation \eqref{dw1} and equation \eqref{dw2},
\[\dim W_{(D_1,...,D_k,E)}=\dim\mathrm{Hom}_H\left(\bigotimes_{i=1}^k \tau^{D_i}_{2n_i} ,\tau^E_{2n}\right)=\dim\mathrm{Hom}_{\mathrm{Sp}_{2m}}\left(\tau^{\iota_{n,m}(E)}_{2m},\bigotimes_{i=1}^k \tau^{\iota_{n_i,m}(D_i)}_{2m}\right ),\]
which completes the proof. $\Box$

\medskip
\begin{cor}\label{duality} 
Let $n,m,k$ be positive integers, $D\in\mathcal{R}_{n,m}$, $J=(j_1,...,j_k)\in\Z^k_{\ge 0}$, $E\in\mathcal{R}_{n+k,m}$,
$\tau^{\iota_{n,m}(D)\otimes (m-J)}_{2m}$ be the representation of $\mathrm{Sp}_{2m}$ given by
\[\tau^{\iota_{n,m}(D)\otimes (m-J)}_{2m} =\tau^{\iota_{n,m}(D)}_{2m}\otimes\left(\bigotimes_{a=1}^k\omega_{m-j_a}\right) \]
where $\omega_1,...,\omega_m$ are the fundamental representations of $\mathrm{Sp}_{2m}$ 
and $\tau^{(D,J)}_H$ be the representation of $H=\mathrm{Sp}_{2n}\times\mathrm{Sp}^k_2$ given by
\[\tau^{(D,J)}_H=\tau^D_{2n}\otimes\left( \bigotimes_{a=1}^k\tau^{(j_a)}_{2}\right)\]
Then
\[\dim\mathrm{Hom}_{\mathrm{Sp}_{2m}}(\tau^{\iota_{n+k,m}(E)}_{2m},\tau^{\iota_{n,m}(D)\otimes (m-J)}_{2m} )
 =\dim\mathrm{Hom}_H(\tau^{(D,J)}_H,\tau^E_{2(n+k)}).\]
\end{cor}


\medskip
\section{Skew Pieri rule for $\mathrm{Sp}_{2m}$}

Using Corollary \ref{duality} for $k=1$, we can translate the explicit branching formula proved by J. Lepowsky \cite{Lep} to an explicit skew Pieri rule for  $\mathrm{Sp}_{2m}$.

 \subsection{Explicit branching formula}\label{br}
 
First let us recall the branching from $\mathrm{Sp}_{2n+2}$ to $\mathrm{Sp}_{2n}$.
Let $G$ and $E$ be Young diagrams with $\ell(G)\le n+1$ and $\ell(E)\le n$, 
  $G=(g_1,\ldots,g_{n+1})$ and  $E=(e_1,\ldots,e_n)$. 
We say that $E$ {\it doubly interlaces} $G$ if
$g_i\geq e_i\geq g_{i+2}$ for $i=1,\ldots,n$,
with the convention $g_{n+2}=0;$ and  we write $E\sqsubset\,\sqsubset G,$ or 
$G\sqsupset\,\sqsupset E.$
Then the branching
of $\tau_{2n+2}^G$ to $\mathrm{Sp}_{2n}$ is in general not multiplicity free
and $\tau_{2n}^E$ appears as a $\mathrm{Sp}_{2n}$ subrepresentation of $\tau^G_{2n+2}$ if and only if  $E$ doubly interlaces $G$.
Furthermore, for $i=1,\ldots, n+1$, we define the numbers 
\begin{equation}\label{rho}
\rho_i(G,E)=x_i-y_i
\end{equation} where
$\{x_1\geq y_1\geq x_2\geq y_2\geq \ldots\geq 
x_{n+1}\geq y_{n+1}\}$ is the rearrangement of the sequence
$\{g_1\ldots,g_{n+1},e_1,\ldots,e_{n},0\}
$
into a weakly decreasing sequence. Then it is well-known that
$$\dim \Hom_{\mathrm{Sp}_{2n}} (\tau_{2n}^E,\tau_{2n+2}^G)=\prod_{i=1}^{n+1}(\rho_i(G,E)+1),$$
see \cite[Theorem 8.1.5]{GW}.

To formulate the explicit formula for branching from
$\mathrm{Sp}_{2n+2}$ to $\mathrm{Sp}_{2n}\times \mathrm{Sp}_2$  we need more notation. 
Let $\{v_1,\ldots,v_{n+1}\}$ be the standard basis for
$\mathbb{R}^{n+1}$ and let
$$\Sigma_{n+1}:=\{v_1\pm v_{n+1},\ldots, v_{n}\pm v_{n+1}\}.$$
For a vector $v$ in $\mathbb{Z}^{n+1}$, define the Kostant partition function $\mathcal{P}_{n+1}(v)$ as the number of ways of writing
$$
v=\sum_{w\in\Sigma_{n+1}}c_w\;w,\ c_w\in\mathbb{N}. 
$$
Then J. Lepowsky 
\cite{Lep} proved the following result.

\medskip
\begin{thm} 
\label{Lep} If $G$ is a Young diagram with $\ell(G)\le n+1$, then we have 
\begin{equation}\label{eq_branching}
\tau_{2n+2}^G\cong \bigoplus_{{E},\,{E\,\sqsubset\,\sqsubset\, G}\,}
 {\bf\rm  m}^G_{E,(\ell)} \tau_{2n}^E\otimes  \tau_2^{(\ell)}
\end{equation}
where the multiplicities ${\bf\rm  m}^G_{E,(\ell)}$ are given by
$$
 {\bf\rm  m}^G_{E,(\ell)}=
\mathcal{P}_{n+1}(\rho_1v_1+\ldots+\rho_{n+1}v_{n+1}-\ell v_{n+1})-
\mathcal{P}_{n+1}(\rho_1v_1+\ldots+\rho_{n+1}v_{n+1}+(\ell+2) v_{n+1}),
$$
and $\rho_i=\rho_i(G,E)$  as in \eqref{rho}.
\end{thm}

In \cite{WY}, N. Wallach and O. Yacobi showed that 
\begin{equation}\label{CGcoef}
 {\bf\rm  m}^G_{E,(\ell)}
=\dim \Hom_{\mathrm{Sp}_2}(\tau_2^{(\ell)},\tau_2^{(\rho_1)}
\otimes\ldots\otimes \tau_2^{(\rho_{n+1})})
\end{equation}
where $\rho_i=\rho_i(G,E)$ 
and formulated the following theorem.
 
\medskip
\begin{thm}{\rm (\cite[Theorem 3.3]{WY})}\label{WY} 
If $G$ is a Young diagram with $\ell(G)\le n+1$, then we have 
 
$$
\tau_{2n+2}^G\cong \bigoplus_{{E},\,{E\,\sqsubset\,\sqsubset\, G}\,}
\tau_{2n}^E\otimes (\tau_2^{(\rho_1)}\otimes\ldots\otimes \tau_2^{(\rho_{n+1})})
$$
as $\mathrm{Sp}_{2n}\times \mathrm{Sp}_2$ modules
where for each $E$ in the sum, $\rho_i=\rho_i(G,E)$  as in \eqref{rho}.
\end{thm}

 \subsection{Explicit formula for the skew Pieri rule}\label{6.3}
 
 The skew duality theorem combined with the multiplicity computations described above yields the following explicit formula
 for the skew Pieri rule.
 
 \begin{thm}\label{thm_pieri}
 Let $\omega_r=\tau_{2m}^{(1^r)}$ be the $r$-th fundamental representation of ${\rm Sp}_{2m}$, $D$  a Young diagram with $\ell(D)\le m$ and $n$  a positive integer such that  $n\ge\ell(D^t)$. Then we have $D\in\mathcal{R}_{m,n}, $ 
  $j_{n,m}(D)\in\mathcal{R}_{n,m}$ and
$$ \tau_{2m}^D\otimes \omega_r
 \cong
 \bigoplus_{
  \begin{array}{c}G\,{\sqsupset\,\sqsupset\, j_{n,m}(D)}\,\\G\in\mathcal{R}_{n+1,m}\end{array}
 }
 {\bf\rm  m}\,^G_{j_{n,m}(D),(m-r)}\;\;\tau_{2m}^{\iota_{n+1,m}(G)}.
 $$
Here $   {\bf\rm  m}\,^G_{j_{n,m}(D),(m-r)}   $ is defined in Theorem \ref{Lep} and the involutions $\iota_{n+1,m}$ and $j_{n,m}$ in Definition \ref{df_involutions}.
 \end{thm}
 
 \medskip
 \begin{proof}
 Let us consider multiplicities $ d^F_{D,r}$ in the decomposition of the tensor
 product defined by the formula
 $$\label{decomp}
 \tau_{2m}^D\otimes\omega_r
 \cong\bigoplus_{F} d^F_{D,r}\,\tau_{2m }^F.
 $$
 Since $D\in\mathcal{R}_{m,n}, $ the multiplicity
$ d^F_{D,r}$ can be strictly positive only if
  $F\in\mathcal{R}_{m,n+1}.$ Indeed,
 it is well known that any highest weight in the decomposition of the tensor product
 is a sum of the highest weight of the first factor and a weight 
of the second one.
But all components of all weights in a fundamental representation of $\mathrm{Sp}_{2m}$ are less or equal to one.
Suppose now that $d^F_{D,r}>0$ for some $F\in\mathcal{R}_{m,n+1}$ and
 set $G=j_{n+1,m}(F),$ so   $F=\iota_{n+1,m}(G).$
 
 We know that the module
 $\tau_{2n}^{j_{n,m}(D)}$ can appear in the branching
 of $\tau_{2n+2}^G$ if and only if 
 $j_{n,m}(D)$ doubly interlaces $G.$
 Finally, by Corollary \ref{duality} for $k=1$ and Theorem \ref{Lep}, we have
 $$
 d^F_{D,r}=d^{\iota_{n+1,m}(G)}_{D,r}=
 {\bf\rm  m}\,^G_{j_{n,m}(D),(m-r)},
 $$
which completes the proof.
   \end{proof}
 
 \subsection{Example}
 
 We shall apply Theorem \ref{thm_pieri} to decompose the tensor product
 $\tau_{2m}^D\otimes\omega_r$ for the
 case $D=(52)$ and $r=4$, when $m=5.$
  Note that
 this is a case where the stable range formula from 
 \cite[2.1.3]{HTW} does not apply. To respect the condition $D\in\mathcal{R}_{m,n},$ 
  we choose $n=5$ (but the computation is similar for
  any choice of  $n$ with $n\geq d_1=5$).
 
 By (\ref{j}),  we get $E:=j_{5,5}(D)=(44433).$
 The conditions on ${G=(g_1,\ldots, g_6)}$, that is,  
$$G\,\sqsupset\,\sqsupset\, E\mbox{\ \
 and\ \ }G\in\mathcal{R}_{6,5},$$
 give the following possibilities for $G:$
 $5\geq g_1\geq g_2\geq 4$, $g_3=4$, $4\geq g_4\geq g_5\geq 3$, $j:=g_6\leq 3.$
 Thus there are 9 choices for $(g_1,\ldots, g_5),$ and for each such choice,
 there are four possible values for $j.$
 
For example, suppose that $G=(44433j).$
Then we get $\rho=(\rho_1,\ldots,\rho_6)=(00000j)$ where $\rho_i=\rho_i(G,E)$ are as in \eqref{rho}.
Using the Clebsch-Gordan formula and \eqref{CGcoef},
we show that ${\bf\rm  m}^G:={\bf\rm  m}^G_{E,(1)}$ is nontrivial just for $j=1$. Moreover, for $j=1$, we have ${\bf\rm  m}^G=1$ and $\iota_{6,5}(G)=(63110)$ by (\ref{i}).

The  most interesting case is $G=(54443j).$
Then $\rho=(10010j)$ and, using the Clebsch-Gordan formula and \eqref{CGcoef}, it is easy to see that
${\bf\rm  m}^G=2$ for $j=1,$
and ${\bf\rm  m}^G=1$ for $j=3.$ Otherwise, ${\bf\rm  m}^G=0$.
Indeed, we have
$$\tau^{(1)}_2\otimes\tau^{(1)}_2\otimes\tau^{(1)}_2=\tau^{(3)}_2\otimes 2\tau^{(1)}_2\mbox{\ \ and\ \ }\tau^{(1)}_2\otimes\tau^{(1)}_2\otimes\tau^{(3)}_2=\tau^{(5)}_2\otimes2\tau^{(3)}_2\otimes\tau^{(1)}_2.$$
The resulting summands $\iota(G)=\iota_{6,5}(G)$ for $j=1,3$ in the decomposition are then
$$
2\,(52110)\oplus (52000).
$$
Here we write $F$ for $\tau^F_{2m}$ as usual.
All results are summarized in the table below. The decomposition
of the tensor product $(52)\otimes(1111)$ of $\mathrm{Sp}_{10}$-modules
has $15$ summands (including multiplicities).

 \medskip

 \begin{tabular}{p{15mm}p{15mm}p{5mm}p{5mm}|p{15mm}||p{15mm}p{15mm}p{5mm}p{5mm}|p{15mm}}
\midrule
$\hspace{5mm}G$ &$\hspace{5mm}\rho$ &$j$ &${\bf\rm  m}^G$  &$\hspace{3mm}\iota(G)$
&$\hspace{5mm}G$ &$\hspace{5mm}\rho$ &$j$ &${\bf\rm  m}^G$  &$\hspace{3mm}\iota(G)$\\
\midrule
(44433j) &(00000j) &1  &1 &(63110)&(44444j) &(00000j) &1 &1 &(61110)\\
\midrule
(54433j) &(10000j) &0 \;2 &1 \;1 &(53111) (53100)&(54444j) &(00000j) &0 \;2 &1 \;1  &(51111) (51100)\\
\midrule
(55433j) &(00000j) &1  &1 &(43110)&(55444j) &(00000j) &1 &1 &(41110)\\
\midrule
(44443j) &(00010j) &0 \;2 &1 \;1  &(62111) (62100)
&(55443j) &(00010j) &0 \;2 &1 \;1  &(42111)  (42100)\\
\midrule
(54443j) &(10010j) &1 &2 &(52110) &(54443j) &(10010j) &3  &1  &(52000)\\
\end{tabular}
 
 \medskip
 A mild extension of this computation shows that the product $(k\ell)\otimes\omega_r$
of $\mathrm{Sp}_{2m}$ modules has always at most 16 summands and enables one to write
down their explicit form.

\subsection{The skew Pieri rule in the stable range} 

The results of Subsection \ref{6.3} can be compared
with those obtained in the stable range in \cite{HTW}. 
In \cite{{HTW}} there is the following formula for a~decomposition of the tensor product of $\mathrm{Sp}_{2m}$-modules in terms of the classical Littlewood-Richardson coefficients. 
Denote by $\rho_m^F$ the irreducible module of the general linear group ${\rm GL}_m={\rm GL}_m(\C)$ corresponding to a~Young diagram $F$ with $\ell(F)\leq m$.
Then the Littlewood-Richardson coefficient $c^{F}_{D,E}$ is defined as the multiplicity
of the ${\rm GL}_m$-module $\rho_m^F$ 
in the tensor product
$\rho_m^D\otimes \rho_m^E.$

\begin{prop}\label{Howe}{\rm(\cite[2.1.3]{HTW})}
Let $D,E$ be Young diagrams with
$\ell(D),\ell(E)\leq m$. Assume that the stable range condition $\ell(D)+\ell(E)\leq m$ holds.
Then 
\begin{equation}\label{eq_stable}
\tau_{2m}^D\otimes \tau_{2m}^E\cong
\bigoplus_{F,\; \ell(F)\leq m}d^F_{D,E}\tau_{2m}^F
\end{equation}
where
$$d^F_{D,E}=\sum_{G_1,G_2,G_3}^{}
c^{F}_{G_1,G_2}
c^{D}_{G_1,G_3}
c^{E}_{G_2,G_3}
$$
where  the sum is taken over all Young diagrams $G_1,G_2,G_3$ with
$\ell(G_1),\ell(G_2),\ell(G_3)\leq m.
$ 
\end{prop}

There is an interesting way to reformulate
the result at least for the case of the tensor product $\tau_{2m}^D\otimes\omega_r$.
  Before stating it, we introduce certain notation. If $D=(d_1,...,d_m)$ and
$E=(e_1,...,e_m)$ are Young diagrams such that $d_i\geq e_i$ for all $i$, then we write $E\subset D$. In this case,  we define the skew diagram $D/E$ to be the diagram obtained by removing all boxes belonging to $E$ from $D$. 
The size of $D/E$ is $|D/E|=|D|-|E|$. For example, if $D=(6,6,5,3,2)$ and $E=(4,4,2,1)$, then $D/E$ is the following skew diagram:
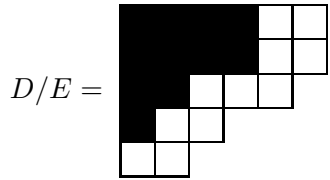
\[D/E=\  \tableau[s]{\fl&\fl&\fl&\fl&&\\
\fl&\fl&\fl&\fl&&\\
\fl&\fl&&&\\
\fl&&\\
&}
\]

Finally, we call the skew diagram $D/E$ a~vertical strip if $d_i-e_i\leq 1$ for all $i$.   
Then it is a~simple observation that the statement of Proposition \ref{Howe} for $\tau_{2m}^D\otimes\omega_r$ is equivalent to the following result. 

\begin{prop}\label{stable}
Let $D$ be a~Young diagram with $\ell(D)\leq m$ and let $\omega_r=\tau_{2m}^{(1^r)}$. 
Assume that the stable range condition $r+\ell(D)\leq m$ holds.  Then 
for a Young diagram $F$ with $\ell(F)\le m$, the multiplicity of $\tau^F_{2m}$ in
the tensor product $\tau_{2m}^D\otimes\omega_r$ is equal to the number of Young diagrams $E$   satisfying the following three conditions:
\begin{itemize}
\item[(i)] $E\subset D$ and $E\subset F$.
\item[(ii)]  $D/E$ and $F/E$ are both vertical strips.
\item[(iii)] $|D/E|+|F/E|=r$.
\end{itemize}
\end{prop}
 
In Proposition \ref{stable}, condition (iii) says that any $F$ is obtained from $D$ by first  removing $i$ boxes and then adding $j$ boxes for some $i$ and $j$ such that  $i+j=r$.
Condition (ii) means that we can add or remove at most one box in each row. 

It turns out that  the result of Proposition \ref{stable} can be extended to one more case outside the stable range.  

\begin{prop}\label{semistable}
	The statement of Proposition \ref{stable} holds true even under
	a~weaker condition $r+\ell(D)\leq m+1$.  
\end{prop}

 \begin{proof}
This follows directly from \cite{KT}.
Indeed, for a Young diagram $F$, denote by $s_{\langle F\rangle}$ the corresponding symplectic Schur function.  
In \cite{KT}  (see also \cite[Proposition 2.2]{Oka}), a~ring homomorphism $\pi_{\mathrm{Sp}_{2m}}$ from the symmetric functions to the character ring of $\mathrm{Sp}_{2m}$ is constructed such that $\pi_{\mathrm{Sp}_{2m}}(s_{\langle F\rangle})$ is the irreducible character for $F$ when $\ell(F)\leq m$, and 
$\pi_{\mathrm{Sp}_{2m}}(s_{\langle F\rangle})=0$ when $\ell(F)= m+1$. Furthermore, it is known that,  for any Young diagrams $D,E$,
\begin{equation}\label{eq_symmetric}
s_{\langle D\rangle}\cdot s_{\langle E\rangle}=
\bigoplus_{F}
\left( \sum_{G_1,G_2,G_3}^{}
c^{F}_{G_1,G_2}
c^{D}_{G_1,G_3}
c^{E}_{G_2,G_3}\right)
s_{\langle F\rangle}
\end{equation}
where  the sums are taken over all Young diagrams $F,G_1,G_2,G_3$
(see \cite[Proposition 2.3]{Oka}).  
Finally, applying the ring homomorphism $\pi_{\mathrm{Sp}_{2m}}$ to \eqref{eq_symmetric}, we show easily that the formula \eqref{eq_stable} holds true   
for the tensor product $\tau_{2m}^D\otimes\omega_r$ even when $r+\ell(D)\leq m+1$. In fact, in this case the coefficient of $s_{\langle F\rangle}$ is non-zero only if $\ell(F)\leq m+1$. See \cite{KT,Oka} for more details.
\end{proof}

A result analogous to Proposition \ref{semistable} was obtained in \cite[Theorem 4.4]{Sun} for the tensor product $\tau_{2m}^D\otimes\bigwedge^r(\C^{2m})$ of $\tau^D_{2n}$ with any antisymmetric power $\bigwedge^r(\C^{2m})$ of the defining representation for $\mathrm{Sp}_{2m}$ (see also \cite[Theorem 4.1]{Oka}). In this case, the submodules of the tensor product are $\tau_{2m}^F$ for $F$ made from $D$ by first adding $i$ boxes and then removing $j$ boxes for some $i,j$ such that $i+j=r$.
From the fact that $\omega_r\cong \bigwedge^r(\C^{2m})/\bigwedge^{r-2}(\C^{2m})$, one can obtain the decomposition of  $\tau_{2m}^D\otimes\omega_r$ in the general case.

\end{document}